\documentclass[a4paper,twoside]{amsart}
\usepackage{xypic,amscd,amssymb,latexsym,srcltx}
\usepackage{dsfont}
\usepackage{amsthm}
\def\@setcopyright{}
\makeatother

\newtheorem{lemma}{Lemma}[section]
\newtheorem{proposition}[lemma]{Proposition}

\newtheorem{theorem}[lemma]{Theorem}
\newtheorem{remark}[lemma]{Remark}
\newtheorem{conjecture}[lemma]{Conjecture}
\newtheorem*{acknowledgments*}{ACKNOWLEDGMENTS}
\newtheorem{definition}[lemma]{Definition}

\newtheorem*{conj*}{Conjecture}
\newtheorem*{thm*}{Main Theorem}
\newtheorem*{remark*}{Remark}
\newtheorem*{lemma*}{Lemma}
\newtheorem*{example*}{[Example]}

\begin{document}
\begin{center}
{\Large \bf Bouc's conjecture on $B$-groups
}\\

\bigskip
\end{center}

\begin{center}
Xingzhong Xu$^{1, 2}$, Jiping Zhang$^{3}$

\end{center}

\footnotetext {$^{*}$~~Date: 14/10/2016.
}
\footnotetext {
1. Department of Mathematics, Hubei University, Wuhan, 430062, China

2. Departament de Matem$\mathrm{\grave{a}}$tiques, Universitat Aut$\mathrm{\grave{o}}$noma de Barcelona, E-08193 Bellaterra,
Spain

~$~$E-mail address: xuxingzhong407@mat.uab.cat; xuxingzhong407@126.com

3. School of Mathematical Sciences, Peking University, Beijing, 100871, China

~$~$E-mail address: jzhang@pku.edu.cn

Supported by National 973 Project (2011CB808003) and NSFC grant (11371124, 11501183).
}

\begin{center}
\footnotesize{\emph{
}}\\
\end{center}

\title{}
\def\abstractname{\textbf{Abstract}}

\begin{abstract}\addcontentsline{toc}{section}{\bf{English Abstract}}
Bouc proposed the following conjecture: a finite group $G$ is nilpotent if
and only if its largest quotient $B$-group $\beta(G)$ is nilpotent. And he has prove that this conjecture holds when
$G$ is solvable. In this paper, we consider the case when $G$ is not solvable. Let
$S$ be a nonabelian simple group except the Chevalley groups
$A_{n}(q)$, $D_{n}(q)$, $E_{6}(q)$, and
$^2A_{n}(q)$, if there exists only one factor of $G$ which
is isomorphic to $S$,
then $\beta(G)$ is not solvable, of course, is not nilpotent. That means we prove the conjecture in these cases.
\hfil\break


\end{abstract}

\maketitle

\section{\bf Introduction and main results}

We all know that the Burnside ring(algebra) is an elementary structure in mathematics, because it
is formed only by the elementary operations on sets: product and coproduct(i.e., disjoint union).
Let $G$ be a finite group, we can do these operations on $G$-sets and get a ring $B(G)$ which is named
as the Burnside ring. And the most important elements of this ring are units and idempotents.
Here, we only focus on the idempotents, the forms of these idempotents can be given by \cite{Gl, Y}.

Because we have some operations from finite groups, thus we can do these operations on Burnside rings.
And these operations are called as
"Induction functor, Restriction functor, Inflation functor, Deflation functor and others functors."
There exists an interesting phenomenon: the Deflation functor can send some idempotent to $0$.
A special class of finite groups which have this interesting phenomenon is defined
as $B$-groups \cite{Bo1,Bo2}.
Indeed, these groups plays an important role in study of Burnside functors.

In \cite{Bo1,Bo2}, Bouc told us that there exists a largest quotient of a finite group which is a $B$-group, of
course, well-defined up to isomorphism, and denoted by $\beta(G)$.
We only know a few properties of $B$-groups, until \cite{Ba} proved some relations between $\beta(G)$ and
$G$. Moreover, \cite{Bo1} proposed the following conjecture:

\begin{conjecture}\cite[Conjecture A]{Bo1} Let $G$ be a finite group. Then $\beta(G)$ is nilpotent if
and only if $G$ is nilpotent.
\end{conjecture}

And \cite{Bo1} gave a proof of this conjecture when $G$ is solvable. Moreover, in \cite{Bo1}, Bouc said that
Th$\mathrm{\acute{e}}$venaz proposed the another conjecture:

\begin{conjecture}\cite[Conjecture]{Bo1} Let $G$ be a finite group. Then $\beta(G)$ is solvable if
and only if $G$ is solvable.
\end{conjecture}

After recalling the basic definitions and properties of $B$-groups, we will give
a proof of the Conjecture 1.1-2 when the group have only one factor which
is isomorphic to some finite nonabelian simple group as following.

\begin{thm*} Let $G$ be a finite group. Let $S$ be a nonabelian simple group except the following groups

(1) $A_{n}(q)$ for $q=p^{f}$ when $n+1\gneq 2$;

(2) $D_{n}(q)$ for $q=p^{f}$ when $n=4$ or $p=2$;

(3) $E_{6}(q)$ for $q=p^f$;

(4) $^2A_{n}(q)$ for $n\gneq 1$.

If there exists only one factor of $G$ which
is isomorphic to $S$,
then $\beta(G)$ is not solvable.
\end{thm*}
Here, the group $G$ can have other nonabelian simple groups as its factors.

\section{\bf Burnside rings and $B$-groups}

In this section we collect some known results that will be needed later.  For the background theory of Burnside rings and
$B$-groups, we refer to \cite{Bo1}, \cite{Bo2},\cite{Dr}

\begin{definition}\cite[Notation 5.2.2]{Bo2} Let G be a finite group and $N\unlhd G$. Denote by $m_{G,N}$ the rational number defined by:
$$m_{G,N}=\frac{1}{|G|}\sum_{XN=G} |X|\mu(X, G),$$
where $\mu$ is the $\mathrm{M\ddot{o}bius}$ function of the poset of subgroups of $G$.
\end{definition}

The reason to definite $m_{G,N}$ is
$$\mathrm{Def}^G_{G/N}e_G^G=m_{G,N}e_{G/N}^{G/N}$$
by \cite{Bo1, Bo2}. Hence, if $m_{G,N}=0$, then $\mathrm{Def}^G_{G/N}e_G^G=0$.
Here, $e_G^G, e_{G/N}^{G/N}$ are idempotents of the Burnside rings of $G$ and $G/N$ respectively.
For the idempotents of  the Burnside rings, we refer to \cite{Gl, Y}.

\begin{remark} If $N=1$, we have
$$m_{G,1}=\frac{1}{|G|}\sum_{X1=G} |X|\mu(X, G)=\frac{1}{|G|}|G|\mu(G, G)=1\neq 0.$$
\end{remark}

\begin{definition}\cite[Definition 2.2]{Bo1} The finite group $G$ is called a $B$-group if
$m_{G,N}=0$ for any non-trivial normal subgroup $N$ of $G$.
\end{definition}

\begin{proposition}\cite[Proposition 5.4.10]{Bo2} Let $G$ be a finite group. If $N_{1}, N_{2}\unlhd G$
are maximal such that $m_{G,N}\neq 0$, then $G/N_{1}\cong G/N_{2}$.
\end{proposition}

\begin{definition}\cite[Notation 2.3]{Bo1} When $G$ is a finite group, and $N\unlhd G$
is maximal such that $m_{G,N}\neq 0$, set $\beta(G)=G/N.$
\end{definition}

\begin{theorem}\cite[Theorem 5.4.11]{Bo2} Let $G$ be a finite group.

1. $\beta(G)$ is a $B$-group.

2. If a $B$-group $H$ is isomorphic to a quotient of $G$, then $H$ is isomorphic to a quotient of $\beta(G)$.

3. Let $N\unlhd G$. The following conditions are equivalent:

\quad\quad (a) $m_{G,N}\neq 0$.

\quad\quad (b) The group $\beta(G)$ is isomorphic to a quotient of $G/N$.

\quad\quad (c) $\beta(G)\cong \beta(G/N)$.
\end{theorem}

We collect some properties of $m_{G,N}$ that will be needed later.

\begin{proposition}\cite[Proposition 2.5]{Bo1} Let $G$ be a finite group.
Then $G$ is a $B$-group if and only if $m_{G,N}=0$ for any minimal (non-trivial) normal subgroup of $G$.
\end{proposition}

\begin{proposition}\cite[Proposition 5.6.1]{Bo2} Let $G$ be a finite group. Then
$m_{G,G}=0$ if and only if $G$ is not cyclic.
If $P$ be cyclic of order $p$ and $p$ be a prime number,
then $m_{P, P}=\frac{p-1}{p}$.
\end{proposition}

\begin{remark} If $G$ is a finite simple group, then $G$ is a $B$-group if and only if
$G$ is not abelian.
\end{remark}

\begin{proposition}\cite[Proposition 5.3.1]{Bo2} Let $G$ be a finite group. If
$M$ and $N$ are normal subgroup of $G$ with $N\leq M$, then
$$m_{G,M}=m_{G,N}m_{G/N, M/N}.$$
\end{proposition}

\section{\bf Some results of the Conjectures}

The Conjecture have been proven in some special cases. We collect two results that will be needed later.

When $p$ is a prime number, recall that a finite group $G$ is called cyclic modulo $p$ (or $p$-hypo-elementary) if
$G/O_{p}(G)$ is cyclic. And M. Baumann has proven the Conjecture under the additional assumption that finite group $G$ is cyclic modulo $p$ in \cite{Ba}.

\begin{theorem}\cite[Theorem 3]{Ba} Let $p$ be a prime number and $G$ be a finite group. Then $\beta(G)$ is cyclic modulo $p$ if
and only if $G$ is cyclic modulo $p$.
\end{theorem}

The reason to consider the $p$-hypo-elementary groups is the author using the useful theorem by Conlon \cite[Lemma 27]{Ba}.
For the Conlon theorem, we refer to \cite[(80.51)]{CR}.

In \cite{Bo1}, S. Bouc has proven the Conjecture under the additional assumption that finite group $G$ is solvable.

\begin{theorem}\cite[Theorem 3.1]{Bo1} Let $G$ be a solvable finite group. Then $\beta(G)$ is nilpotent if
and only if $G$ is nilpotent.
\end{theorem}

\section{\bf Proof of the Main Theorem}

In this section, we will prove the main theorem.
\begin{theorem} Let $G$ be a finite group. Let $S$ be a non-abelian simple group and $\mathrm{Out}(S)=\mathrm{Aut}(S)/\mathrm{Inn}(S)$ be cyclic modulo $p$ (or $p$-hypo-elementary). If the group $G$ has only one factor which
is isomorphic to $S$,
then $\beta(G)$ is not solvable.
\end{theorem}

\begin{proof}Let $G$ be a minimal counterexample, that is $\beta(G)$ is solvable, and,
if $L$ is a finite group with $|L|\lneq |G|$ and, $L$ has only one factor which
is isomorphic to $S$, then $\beta(L)$ is not solvable.

Let $M\unlhd G$ is maximal such that $m_{G, M}\neq 0$, then $\beta(G)=G/M$ is solvable.
Here, $M\neq 1$ because $G$ is not solvable.
Now we can assume that $1\neq N\leq M$ is a minimal normal subgroup of $G$, thus we have
$$m_{G, M}=m_{G, N}m_{G/N, M/N}$$
by \cite[Proposition 5.3.1]{Bo2}.
Thus $m_{G, N}\neq 0$. So by \cite[Proposition 5.4.11]{Bo2}, thus $\beta(G/N)=\beta(G)$ is solvable.
Since $N$ is a minimal normal subgroup of $G$, thus $N$ is abelian, or
$$N\cong \underbrace{T\times T\times \cdots \times T}_k.$$
Here, $T$ is a nonabelian simple group.

Since $\beta(G/N)$ is solvable and $|G/N|\lneq |G|$, thus $G/N$ does not have a one factor which
is isomorphic to $S$. That is $N$ has only a one factor which
is isomorphic to $S$, it implies that
$$N\cong T\cong S.$$

We will consider the cases whether $C_{G}(N)=1$ in the following.

\textbf{Case 1.} If $C_{G}(N)\neq 1$, thus $|G/C_{G}(N)|\lneq |G|$ and $G/C_{G}(N)$ has only one factor which
is isomorphic to $S$. Also $\beta(G/C_{G}(N))$ is a $B$-group and is a quotient of $G/C_{G}(N)$,
then $\beta(G/C_{G}(N))$ is isomorphic
to a quotient of $\beta(G)$ by \cite[Theorem 2.4.2]{Bo1}. So $\beta(G/C_{G}(N))$ is solvable.
That is a contradiction to the minimal counterexample $G$.

\textbf{Case 2.} If $C_{G}(N)=1$, we can see that
$$N\unlhd G\leq \mathrm{Aut}(N).$$
Here, $N\cong S$. Since $\mathrm{Out}(S)=\mathrm{Aut}(S)/\mathrm{Inn}(S)$ is cyclic modulo $p$ for some prime number $p$,
 thus $G/N$ is also cyclic modulo $p$ because $G/N$ is isomorphic to a subgroup of $\mathrm{Aut}(S)/\mathrm{Inn}(S)$.
 Since $m_{G,N}\neq 0$, thus $\beta(G)$ is cyclic modulo $p$.
It implies $G$ is cyclic modulo $p$ by \cite{Ba}, that is $G$ is solvable. That is a contradiction.

Therefore, we have $\beta(G)$ is not solvable.
\end{proof}

Now, we will only need to prove that the outer automorphism of almost nonabelian simple groups
 are cyclic modulo $p$ for some prime number $p$.
The Theorem 4.2 is about the case when the simple groups are alternating groups and sporadic groups. And
the Theorem 4.3 is about the case when the simple groups are simple groups of Lie type.
\begin{theorem} Let $S$ be a nonabelian simple, if $S$ is type of alternating group or sporadic group,
then $\mathrm{Out}(S)$ is cyclic modulo $p$ for some prime number $p$.
\end{theorem}

\begin{proof}By\cite{GLS3}, we have $|\mathrm{Out}(S)|=2 $ or $1$ when $S$ is a sporadic group. Hence, $\mathrm{Out}(S)$
is cyclic modulo $p$ when $S$ is a sporadic group.
If $S:=A_n$ is an alternating group with degree $n$, then
$$|\mathrm{Out}(S)|=\left\{ \begin{array}{ll}
2, &
\mbox{if}~ n\geq 5 ~and~ n\neq 6;
\\[2ex] 4, &\mbox{if} ~n=6.\end{array}\right.$$

Hence $\mathrm{Out}(S)$ is a $2$-group, of course, $S$ is cyclic modulo $2$.
\end{proof}

First, we give the following notations.
Let $n$ be a positive integer, $q \gneq 1$ be a power of a prime number $p$, and $q$ be the order of some underlying finite field.
The notation $(a,b)$ represents the greatest common divisor of the integers $a$ and $b$. And $a|b$ represents $a$ is a factor of $b$.

\begin{theorem} Let $S$ be a nonabelian simple, if $S$ is type of Lie type except the following groups

(1) $A_{n}(q)$ for $q=p^{f}$ when $n+1\gneq 2$;

(2) $D_{n}(q)$ for $q=p^{f}$ when $n=4$ or $p=2$;

(3) $E_{6}(q)$ for $q=p^f$;

(4) $^2A_{n}(q)$ for $n\gneq 1$,

then $\mathrm{Out}(S)$ is cyclic modulo $r$ for some prime number $r$.
\end{theorem}

\begin{proof}
By \cite[p.xv]{CCNPW}, the outer automorphism group of the simple Chevalley group is a semidirect product (in this order)
 of groups of orders $d$ (diagonal automorphisms),  $f$ (field automorphisms, generated by a Frobenius automorphism by\cite[p.246, Theorem 5.4]{L}), and $g$ (graph automorphisms modulo field automorphisms, coming from automorphisms of the Dynkin diagram), except
 that for
 $$B_2(2^f), G_2(3^f), F_4(2^f)$$
 the (extraordinary) graph automorphism squares to the generating field automorphism.
 By \cite[p.xvi, Table 5]{CCNPW}, we have
\[
\begin{array}{|c|c|c|c|c|c|}\hline
\mathrm{Type} & \mathrm{Condition} &\mathrm{Name}&d &f &g \\\hline
1 & & A_{1}(q)& (2, q-1)& q=p^f & 1\\\hline
2 & n\geq 2 & A_{n}(q)& (n+1, q-1)& q=p^f & 2\\\hline
3 & n\geq 2 & ^2A_n(q)& (n+1, q+1)& q^2=p^f & 1\\\hline
4 &   & B_2(q)& (2, q-1)& q=p^f & 2~\mathrm{if}~p=2\\\hline
5 & f~ \mathrm{odd} & ^2B_2(q)& 1 & q=2^f & 1\\\hline
6 &  n\geq 3 & B_n(q)& (2, q-1)& q=p^f & 1\\\hline
7 &  n\geq 3 & C_n(q)& (2, q-1)& q=p^f & 1\\\hline
\end{array}
\]
$$\mathrm{Table}~1$$
\[
\begin{array}{|c|c|c|c|c|c|}\hline
\mathrm{Type} & \mathrm{Condition} &\mathrm{Name}&d &f &g \\\hline
8 & & D_{4}(q)& (2, q-1)^2 & q=p^f & 3!\\\hline
9 &  & ^3D_{4}(q)& 1 & q^3=p^f & 1\\\hline
10 & n\gneq 4~\mathrm{even} & D_n(q)& (2, q-1)^2& q=p^f & 2\\\hline
11 &  n\gneq 4~\mathrm{odd}  & D_n(q)& (4, q^n-1)& q=p^f & 2\\\hline
12 & n\geq 4 &  ^2D_n(q)&(4, q^n+1) & q^2=p^f & 1\\\hline
\end{array}
\]
$$\mathrm{Table}~2$$
\[
\begin{array}{|c|c|c|c|c|c|}\hline
\mathrm{Type} & \mathrm{Condition} &\mathrm{Name} & d &f &g \\\hline
13 & & G_{2}(q)& 1& q=p^f & 2~\mathrm{if}~ p=3\\\hline
14 &f~\mathrm{odd} & ^2G_{2}(q)& 1& q=3^f & 1\\\hline
15 &   & F_4(q)& 1 & q=p^f & 2~\mathrm{if}~ p=2\\\hline
16 &f~\mathrm{odd}  & ^2F_4(q)& 1 & q=2^f & 1\\\hline
17 &  & E_6(q)& (3, q-1) & q=p^f & 2\\\hline
18 &   & ^2E_6(q)& (3, q+1)& q^2=p^f & 1\\\hline
19 &   & E_7(q)& (2, q-1)& q=p^f & 1\\\hline
20 &  & E_8(q)& 1 & q=p^f & 1\\\hline
\end{array}
\]
$$\mathrm{Table}~3$$
In the Table 1-3, the groups of order $d,f,g$ are cyclic except that
$3!$ indicates the symmetric group of degree $3$. And an entry '$2$ if $\ldots$' means $1$ if not.

 \textbf{ Case 1.} In the Table 1-3, we can find the outer automorphism group of $S$ is cyclic when
 $S$ is of the type 5, 9, 14, 16, 20. Hence, in these types, $\mathrm{Out}(S)$ is cyclic mod $r$ for some prime $r$.

  \textbf{ Case 2.} In the Table 1-3, we can find the outer automorphism group of $S$ is a semidirect product of groups of orders $d$ and $f$ when
 $S$ is of the type 1, 3, 6, 7, 12, 18, 19. Since $d$ is prime power for some prime number except the type 3, we can see $\mathrm{Out}(S)$ is cyclic mod $r$ for some prime $r$ in these cases except the
 type 3.

\textbf{ Case 3.} In the Table 1-3, for the type 4  at $p=2$, the type 13 at $p=3$ and the type 15 at $p=2$, we have the outer automorphism group of $S$ is
cyclic.  Hence, in these types, $\mathrm{Out}(S)$ is cyclic mod $r$ for some prime $r$.

For the type 4 at $p\neq 2$, since an entry '$2$ if $\ldots$' means $1$ if not, thus the outer automorphism group of $S$ is a semidirect product of groups of orders $d$ and $f$ and $d|2$. Hence, we have $\mathrm{Out}(S)$ is cyclic mod $2$ in this case.

For the type 13 at $p\neq 3$, since an entry '$2$ if $\ldots$' means $1$ if not, thus the only outer automorphisms of $G_2(q)$
  are field automorphisms. Hence, in this type, $\mathrm{Out}(S)$ is cyclic mod $r$ for some prime $r$.

For the type 15 at $p\neq 2$, since an entry '$2$ if $\ldots$' means $1$ if not, thus the only outer automorphisms of $F_4(q)$
 are field automorphisms. Hence, in this type, $\mathrm{Out}(S)$ is cyclic mod $r$ for some prime $r$.

\textbf{ Case 4.} For the type 10 and 11, if $q$ is odd,
by \cite[p.75]{Wi}, we have $\mathrm{Out}(S)$ is cyclic mod $2$ in this case.
If $q=2^f$, we have  $$\mathrm{Out}(S)\cong C_{f}\rtimes C_2.$$
That means $\mathrm{Out}(S)$ is not always cyclic mod $r$ for some prime $r$ when $q=2^f$.

\textbf{ Case 5.} In the Table 1-3, the outer automorphism groups of the type 8 and 17 are not always cyclic mod $r$ for some prime $r$.

For the type 2, by \cite[p.50, Theorem 3.2]{Wi}, if $n+1=2$, then
 $\mathrm{Out}(S)$ is cyclic mod $2$. But if $n+1\gneq 2$, we have
 $\mathrm{Out}(S)\cong D_{2(n+1, q-1)}\times C_f$ where $q=p^f$.
 So, $\mathrm{Out}(S)$ is cyclic mod $r$ for some prime $r$ except $n+1\gneq 2$.

Hence, the outer automorphism groups of above simple groups are  cyclic mod $r$ for some prime $r$ except the following:
$$ A_{n}(q)~ \mathrm{for}~ n+1\gneq 2, ~^2A_{n}(q), D_n(q)~\mathrm{for}~n=4~ \mathrm{or} ~p=2, E_6(q).$$
\end{proof}

\textbf{ACKNOWLEDGMENTS}\hfil\break
The authors would like to thank Prof. S. Bouc for his numerous discussion in Beijing in Oct. 2014.
And the first author would like to thank  Prof. C. Broto for his constant encouragement in Barcelona in Spain.

\end{document}